\documentclass[12pt]{amsart}

\usepackage[T1]{fontenc}
\usepackage{amsmath,amssymb,amsthm}
\usepackage[numeric]{amsrefs}
\usepackage{microtype}
\usepackage[svgnames]{xcolor}
\usepackage[margin=1in]{geometry}
\usepackage[
colorlinks=true,
linkcolor=DarkBlue,
urlcolor=DarkRed,
citecolor=DarkGreen
]{hyperref}

\allowdisplaybreaks

\newtheorem{thm}{Theorem}[section]

\newtheorem{lem}[thm]{Lemma}
\newtheorem{coro}[thm]{Corollary}

\theoremstyle{remark}
\newtheorem{rem}[thm]{Remark}

\DeclareMathOperator{\reg}{reg}
\DeclareMathOperator{\pd}{pd}
\DeclareMathOperator{\depth}{depth}
\DeclareMathOperator{\NM}{NM}
\DeclareMathOperator{\lk}{lk}
\DeclareMathOperator{\st}{st}

\numberwithin{equation}{section}

\begin{document}
\title{Homology of non-matching complexes under edge additions and applications to their Stanley-Reisner ideals}
\def\shorttitle{Homology of non-matching complexes}

\author{Jiawen Shan}
\address{School of Mathematics and Systems Science, Shenyang Normal University, Shenyang 110034, P. R. China}
\email{JiawenShan@synu.edu.cn}

\author{Dancheng Lu}
\address{School of Mathematical Sciences, Soochow University, Suzhou 215006, P. R. China}
\email{ludancheng@suda.edu.cn}
\author[Zexin Wang]{Zexin Wang$^{*}$} 
\address{School of Mathematical Sciences, Soochow University, Suzhou 215006, P. R. China}
\email{zexinwang6@outlook.com}
\thanks{$*$ Corresponding author.}

\begin{abstract}
For a bipartite graph $G$ and an integer $t\geq 2$, we study the $t$-non-matching complex $\NM_t(G)$ and its Stanley-Reisner ideal. We prove that, if adding an edge $e$ to $G$ preserves bipartiteness, the inclusion $\NM_t(G)\hookrightarrow\NM_t(G+e)$ induces an injection on reduced homology in degree $2t-3$.
%We prove that adding an edge while preserving bipartiteness induces an injection on reduced homology in degree $2t-3$. 
If $G$ contains a cycle of length at least $2t$, then $\NM_t(G)$ has Leray number $2t-2$, and its Stanley-Reisner ideal has regularity $2t-1$. We also obtain explicit lower bounds for certain graded Betti numbers and the projective dimension of this ideal. If $G$ contains a $2t$-cycle, we determine the maximal shifts of this ideal through homological degree $|E(G)|-2t+1$. For the complete bipartite graph $G=K_{r, s}$ with $2\leq t\leq r\leq s$, we determine the depth, projective dimension, all maximal shifts, and the unique extremal Betti number of the associated ideal, thereby settling a conjecture on facet ideals of chessboard complexes.	
\end{abstract}

\date{}
\maketitle

\section{Introduction}
Throughout, $K$ denotes a field. Let $G$ be a finite simple graph with vertex set $V(G)$ and edge set $E(G)$. For $W\subseteq E(G)$, let $G_W$ denote the spanning subgraph of $G$ with edge set $W$. A \emph{matching} in $G$ is a set of pairwise disjoint edges, and a \emph{$t$-matching} is a matching of size $t$. The \emph{matching number} $\nu(G)$ is the maximum size of a matching in $G$. For an integer $t \geq 2$, the \emph{$t$-non-matching complex} of $G$ is the simplicial complex on $E(G)$ defined by $$\NM_t(G)=\{W\subseteq E(G):\nu(G_W)<t\}.$$

The topology of non-matching complexes of graphs was studied by Linusson, Shareshian, and Welker \cite{LSW08}, who determined the homotopy types for complete and complete bipartite graphs. Aharoni, Holzman, and Jiang \cite{AHJ19} obtained collapsibility bounds in their work on rainbow fractional matchings. Holmsen and Lee \cite{HL22} later established homology vanishing results for non-matching complexes and links of their nonempty faces. For broader background on simplicial complexes of graphs, including complexes defined by matching restrictions, see \cite{Jonsson08}.

For a simplicial complex $\Delta$ and a subset $W$ of its vertex set, let $\Delta_W$ denote the induced subcomplex of $\Delta$ on $W$. The \emph{Leray number} of $\Delta$ over $K$, denoted by $L_K(\Delta)$, is the smallest integer $d \geq 0$ such that $$\widetilde H_q(\Delta_W; K)=0$$ for every $W$ and every $q \geq  d$. If $G$ is bipartite, then, for every $W\subseteq E(G)$, the graph $G_W$ is a bipartite spanning subgraph of $G$, and $(\NM_t(G))_W=\NM_t(G_W).$ The proof of \cite{AHJ19}*{Theorem~1.7} shows that $\NM_t(G)$ is $(2t-2)$-collapsible, since the fractional and ordinary matching numbers coincide for bipartite graphs. Consequently, $L_K(\NM_t(G))\leq 2t-2.$ This upper bound also follows from \cite{HL22}*{Theorem~1.1}. It is therefore natural to ask when this upper bound is attained. Cycles provide sharp examples for this bound.  Indeed, if $C_\ell$ is a cycle with $\ell\geq 2t$, then by Lemma~\ref{lem:cyclic-core}, we have $\widetilde H_{2t-3}(\NM_t(C_\ell);K)\ne 0$. Hence, $L_K(\NM_t(C_\ell))=2t-2$. This raises a natural question: does this nonvanishing persist when we pass from such a cycle to a more general bipartite graph containing it?

%We organize this paper as follows: Section 2 contains a fundamental exploration on the reduced homology of non-matching complexes under edge additions. We prove a key theorem on the induced injection (Theorem~\ref{thm:edge-addition}) and obtain algebraic consequences for their Stanley-Reisner ideals (Corollary~\ref{cor:algebraic-consequences}). Section 3 is devoted to explicit computations on complete bipartite graphs $K_{r,s}$. We determine the homological invariants of the corresponding ideal $J_{t;r,s}$, including its depth, projective dimension, and extremal Betti numbers (Theorem~\ref{thm:complete-bipartite}), which successfully confirms a recent conjecture posed by Jiang, Zhao, Wang, and Zhu \cite{JZWZ23}.

We prove the following edge-addition result:
if $H$ is a bipartite graph and $e$ is a
nonedge of $H$ such that $H+e$ remains
bipartite, then the inclusion
$\NM_t(H)\hookrightarrow\NM_t(H+e)$
induces an injection on reduced homology
in degree $2t-3$;
see Theorem~\ref{thm:edge-addition}.
Consequently, if a bipartite graph $G$
contains a cycle $C$ of length at least
$2t$, then
$\widetilde H_{2t-3}(\NM_t(G_W);K)\ne 0$
for every $W$ satisfying
$E(C)\subseteq W\subseteq E(G)$.
Together with the upper bound above,
this yields $L_K(\NM_t(G))=2t-2$;
see Corollary~\ref{thm:cycle-persistence}.

We also translate the preceding topological results into algebraic terms. Hochster's formula relates the reduced homology of induced subcomplexes to the multigraded Betti numbers of the corresponding Stanley-Reisner ideal. We first fix the relevant notation.
Set $R=K[x_e:e\in E(G)]$ and endow it with the natural $\mathbb Z^{E(G)}$-grading given by $\deg(x_e)=\varepsilon_e$, where $\{\varepsilon_e:e\in E(G)\}$ is the standard basis of $\mathbb Z^{E(G)}$. For $W\subseteq E(G)$, write $x_W=\prod_{e\in W}x_e$. If $\Delta$ is a simplicial complex on $E(G)$, its \emph{Stanley-Reisner ideal} is $$I_\Delta=(x_W:W\subseteq E(G),\ W\notin\Delta) \subseteq R.$$

For $i \geq 0$ and $W\subseteq E(G)$, the $i$-th \emph{multigraded Betti number} of $I_\Delta$ in multidegree $\mathbf 1_W$ is $$\beta_{i, W}(I_\Delta)=\dim_K \operatorname{Tor}_i^R(I_\Delta,K)_{\mathbf 1_W},$$ where $\mathbf 1_W$ denotes the characteristic vector of $W$. Since $I_\Delta$ is squarefree, for $i, j \geq 0$, its $(i, j)$-th \emph{graded Betti number} is given by $$\beta_{i, j}(I_\Delta)=\sum_{\substack{W\subseteq E(G)\\ |W|=j}}\beta_{i, W}(I_\Delta).$$

Assume now that $I_\Delta\ne 0$. Whenever $\beta_{i, j}(I_\Delta)\ne 0$ for some $j$, the $i$-th \emph{maximal shift} of $I_\Delta$ is $$t_i(I_\Delta)=\max\{j:\beta_{i, j}(I_\Delta)\ne 0\}.$$ The \emph{projective dimension} of $I_\Delta$ is $$\pd I_\Delta=\max\{i:\beta_{i, j}(I_\Delta)\ne 0 \text{ for some }j\}.$$ The \emph{Castelnuovo-Mumford regularity}, or simply the \emph{regularity}, of $I_\Delta$ is $$\reg I_\Delta=\max\{j-i:\beta_{i, j}(I_\Delta)\ne0\} =\max_{0\leq i\leq\pd I_\Delta}\{t_i(I_\Delta)-i\}.$$

We then specialize to the non-matching complex. Since the minimal nonfaces of $\NM_t(G)$ are precisely the $t$-matchings of $G$, its Stanley-Reisner ideal is $$I_{\NM_t(G)}=(x_W: W\text{ is a $t$-matching of }G)\subseteq R.$$ Thus Hochster's formula (see, for instance, \cite{HH11}*{Theorem~8.1.1}) gives
\begin{equation}\tag{H}\label{eq:hochster}
	\beta_{i, W}(I_{\NM_t(G)})=\dim_K \widetilde H_{|W|-i-2}(\NM_t (G_W); K),
\end{equation}
for every $i\geq 0$ and $W\subseteq E(G)$. In particular, whenever $I_{\NM_t(G)}\ne 0$, Hochster's formula yields the standard relation (see also \cite{KM06})
\stepcounter{equation}
\begin{equation}\tag{R}\label{eq:regularity-leray}
	\reg I_{\NM_t(G)}=L_K(\NM_t(G))+1.
\end{equation}

Suppose that $G$ is a bipartite graph containing a cycle $C$ of length $\ell \geq 2t$, and set $n=|E(G)|$. Then $$\reg I_{\NM_t(G)}=2t-1,
\qquad
\pd I_{\NM_t(G)}\geq n-2t+1,$$
and $$\beta_{i, i+2t-1}(I_{\NM_t(G)}) \geq \binom{n-\ell}{i+2t-1-\ell}$$ for $\ell-2t+1\leq i\leq n-2t+1$. In particular, if $G$ contains a $2t$-cycle, we determine the maximal shifts of $I_{\NM_t(G)}$ through homological degree $n-2t+1$; see Corollary~\ref{cor:algebraic-consequences}.

Finally, let $K_{r,s}$ denote the complete bipartite graph with parts of sizes $r$ and $s$. For $2\leq t\leq r\leq s$, we set
$$R_{r, s}=K[x_{ij}:1\leq i\leq r,\ 1\leq j \leq s],
\qquad
J_{t; r, s}=I_{\NM_t(K_{r, s})}.$$
Let $\Delta_{r, s}$ denote the matching complex of $K_{r, s}$, whose faces are the matchings of $K_{r, s}$. This is the chessboard complex; see \cite{BLVZ94} for background. For a simplicial complex $\Gamma$, its \emph{pure $d$-skeleton} $\Gamma^{[d]}$ is the subcomplex generated by its $d$-dimensional faces. Under the standard identification of the edges of $K_{r, s}$ with the squares of an $r \times s$ chessboard, the $t$-matchings correspond precisely to the facets of the pure $(t-1)$-skeleton $\Delta_{r, s}^{[t-1]}$. 
For a simplicial complex $\Gamma$, its
\emph{facet ideal} is
$$F(\Gamma)=(\prod_{v\in F}x_v: F\text{ is a facet of }\Gamma);$$
see \cite{Faridi02}.
 Consequently, $$J_{t; r, s}=F \bigl( \Delta_{r, s}^{[t-1]} \bigr),$$ and, in particular,
$$J_{r; r, s}=F(\Delta_{r, s}).$$ Jiang, Zhao, Wang, and Zhu \cite{JZWZ23} proved that $$\reg\bigl( R_{r, s}/F(\Delta_{r, s}) \bigr)=\depth \bigl(R_{r, s}/F(\Delta_{r, s}) \bigr)=2r-2 $$ when \(r \leq 3\), and proposed a conjecture stating that the same formulas hold for all \(1\leq r\leq s\). 

We establish the corresponding formulas for the pure $(t-1)$-skeleta of chessboard complexes:
$$\reg(R_{r, s}/J_{t; r, s})=\depth(R_{r, s}/J_{t; r, s})=2t-2
\qquad
(2\leq t\leq r\leq s).$$
In particular, taking $t=r$ and using the case $r=1$ established in \cite{JZWZ23}*{Theorem~4.2}, we settle \cite{JZWZ23}*{Conjecture~5.15}. We also determine the projective dimension, all maximal shifts, and the unique extremal Betti number of $J_{t; r, s}$; see Theorem~\ref{thm:complete-bipartite}.

%%%%%%%%%%%%%%%%%%%%%%%%%%%%%%%%%%%%%%%%%%%%%%%%%%%%%%%%%%%%%%%%%%%%%%%%%%%

\section{Homology under edge additions and Stanley-Reisner consequences}
We recall some fundamental concepts. Let $\Delta$ be a simplicial complex. For a face $\sigma\in\Delta$, the \emph{link} and the \emph{star} of $\sigma$ in $\Delta$ are defined by $$\lk_\Delta(\sigma)=\{\tau\in\Delta: \tau\cap\sigma=\varnothing,\ \tau\cup\sigma\in\Delta\}$$ and $$\st_\Delta(\sigma)=\{\tau\in\Delta: \tau\cup\sigma\in\Delta\},$$ respectively. For a vertex $v$ of $\Delta$, we denote $\lk_\Delta(\{v\})$ and $\st_\Delta(\{v\})$ simply by $\lk_\Delta(v)$ and $\st_\Delta(v)$, respectively. In particular, $\st_\Delta(v)$ is a cone with apex $v$. 

We first prove the following edge-addition result, which plays a key role in the persistence argument.

%\begin{thm}\label{thm:edge-addition}
%	Let $t \geq 2$, let $G$ be a bipartite graph, and let $H$ be a spanning subgraph of $G$. For $e\in E(G)\setminus E(H)$, let $H+e$ denote the spanning subgraph of $G$ with edge set $E(H)\cup\{e\}$. Then the inclusion $\NM_t(H)\hookrightarrow\NM_t(H+e)$ induces an injection $$\widetilde H_{2t-3}(\NM_t(H);K) \hookrightarrow \widetilde H_{2t-3}(\NM_t(H+e);K).$$
%\end{thm}

%\begin{proof}
%Define $\Delta_0=\NM_t(H)$ and $\Delta_1=\NM_t(H+e)$. Since $t\geq 2$, the singleton $\{e\}$ is a face of $\Delta_1$. Moreover, $\Delta_0$ is the deletion of the vertex $e$ from $\Delta_1$. Hence
%$$\Delta_1=\Delta_0 \cup\st_{\Delta_1}(e),
%\quad
%\Delta_0\cap\st_{\Delta_1}(e)=\lk_{\Delta_1}(e).$$
%Let $d=2t-3$. The Mayer-Vietoris sequence contains
%$$\widetilde H_d(\lk_{\Delta_1}(e);K) \longrightarrow \widetilde H_d(\Delta_0;K)\oplus \widetilde H_d(\st_{\Delta_1}(e);K) \longrightarrow \widetilde H_d(\Delta_1;K).$$
%Since $\st_{\Delta_1}(e)$ is a cone with apex $e$, it follows that $$\widetilde H_d(\st_{\Delta_1}(e);K)=0.$$ Moreover, $\{e\}$ is a nonempty face of $\Delta_1$, and $H+e$ is bipartite. Thus \cite{HL22}*{Theorem~1.2} yields $$\widetilde H_d(\lk_{\Delta_1}(e);K)=0.$$ 
%Hence, by exactness, the map $$\widetilde H_d(\Delta_0;K)\longrightarrow \widetilde H_d(\Delta_1;K)$$ is injective.
%\end{proof}

\begin{thm}\label{thm:edge-addition}

Let $t\geq 2$, let $G$ be a bipartite graph,
and let $e$ be a nonedge of $G$ such that
$G+e$ is bipartite.
 Then the inclusion $\NM_t(G) \hookrightarrow \NM_t(G+e)$ induces an injection 
	$$\widetilde H_{2t-3}(\NM_t(G);K)
	\hookrightarrow
	\widetilde H_{2t-3}(\NM_t(G+e);K).$$
\end{thm}

\begin{proof}
Define $\Delta_0=\NM_t(G)$ and $\Delta_1=\NM_t(G+e)$. Since $t\geq 2$, the singleton $\{e\}$ is a face of $\Delta_1$. Moreover, $\Delta_0$ is the deletion of the vertex $e$ from $\Delta_1$. Hence
$$\Delta_1=\Delta_0 \cup\st_{\Delta_1}(e),
\quad
\Delta_0\cap\st_{\Delta_1}(e)=\lk_{\Delta_1}(e).$$
Let $d=2t-3$. The Mayer-Vietoris sequence contains
$$\widetilde H_d(\lk_{\Delta_1}(e);K)
\longrightarrow
\widetilde H_d(\Delta_0;K)\oplus
\widetilde H_d(\st_{\Delta_1}(e);K)
\longrightarrow
\widetilde H_d(\Delta_1;K).$$
Since $\st_{\Delta_1}(e)$ is a cone with apex $e$, it follows that $$\widetilde H_d(\st_{\Delta_1}(e);K)=0.$$ 
Moreover, $\{e\}$ is a nonempty face of
$\Delta_1$, and $G+e$ is bipartite.
 Thus \cite{HL22}*{Theorem~1.2} yields $$\widetilde H_d(\lk_{\Delta_1}(e);K)=0.$$ 
Hence, by exactness, the map
$$\widetilde H_d(\Delta_0;K)\longrightarrow \widetilde H_d(\Delta_1;K)$$
is injective.
\end{proof}

The following lemma provides the initial nonvanishing needed for the persistence argument.

\begin{lem}\label{lem:cyclic-core}
	Let $t\geq 2$ be an integer and let $C_\ell$ be a cycle of length $\ell\geq 2t$.
	Then $\widetilde H_{2t-3}(\NM_t(C_\ell);K)\ne 0$.
\end{lem}

\begin{proof}
Label the edges of $C_\ell$ cyclically as $e_1, \ldots, e_\ell$ and identify $e_i$ with $i\in[\ell]$. Then the minimal nonfaces of $\NM_t(C_\ell)$, namely the $t$-matchings of $C_\ell$, correspond exactly to the independent $t$-subsets of the cycle on $[\ell]$.
	
Since $\ell\geq2t$, we have $t-1<\ell/2$. Let $M_{\ell, t-1}$ be the simplicial complex introduced in \cite{CDNW16}, whose faces are the subsets of $[\ell]$ contained in the set of vertices covered by a matching of size $t-1$ in the cycle on $[\ell]$. By \cite{CDNW16}*{Lemma~4.2}, we see that the minimal nonfaces of $M_{\ell, t-1}$ are precisely those independent $t$-subsets of the cycle on $[\ell]$. Hence $$\NM_t(C_\ell)\cong M_{\ell,t-1}.$$ By \cite{CDNW16}*{Lemma~4.1}, $M_{\ell,t-1}$ is isomorphic to the boundary complex of the $(2t-2)$-dimensional cyclic polytope with $\ell$ vertices. This boundary complex is a sphere of dimension $2t-3$, so the result follows.
\end{proof}

Combining Lemma~\ref{lem:cyclic-core} with repeated applications of Theorem~\ref{thm:edge-addition} yields the following persistence result.

\begin{coro}\label{thm:cycle-persistence}
	Let $G$ be a bipartite graph, let $t\geq 2$, and suppose that $G$ contains a cycle $C$ of length $\ell\geq 2t$. Then, for every $W\subseteq E(G)$ with $E(C)\subseteq W$, we have
	$$\widetilde H_{2t-3}(\NM_t(G_W);K)\ne 0
	\quad\text{and}\quad
	L_K(\NM_t(G_W))=2t-2.$$
In particular, $L_K(\NM_t(G))=2t-2$.
\end{coro}

\begin{proof}
	Let $W\subseteq E(G)$ satisfy $E(C)\subseteq W$. Start with the spanning subgraph $G_{E(C)}$ and add the edges of $W\setminus E(C)$ one at a time. Every intermediate graph is a spanning subgraph of the bipartite graph $G$. Since isolated vertices do not affect the non-matching complex, we have $\NM_t(G_{E(C)})=\NM_t(C)$. 
	Applying Theorem~\ref{thm:edge-addition} successively shows that the inclusion $\NM_t(C)=\NM_t(G_{E(C)})\hookrightarrow\NM_t(G_W)$ induces an injection on reduced homology in degree $2t-3$. By Lemma~\ref{lem:cyclic-core}, this gives
	$$0\ne \widetilde H_{2t-3}(\NM_t(C);K)
	\hookrightarrow
	\widetilde H_{2t-3}(\NM_t(G_W);K).$$
It follows that $L_K(\NM_t(G_W))\geq 2t-2$. On the other hand, for every $U\subseteq W$, the graph $G_U$ is bipartite and $(\NM_t(G_W))_U=\NM_t(G_U)$. Thus \cite{HL22}*{Theorem~1.1} gives $L_K(\NM_t(G_W))\leq 2t-2$, and hence $L_K(\NM_t(G_W))=2t-2$.
\end{proof}

We now translate the preceding homological results into algebraic consequences for the Stanley-Reisner ideal of the non-matching complex. For an ideal $I\subseteq R$, we write $\depth I$ for its depth as an $R$-module. The following corollary collects these consequences.

\begin{coro}\label{cor:algebraic-consequences}
	Let $G$ be a bipartite graph, let $t\geq 2$, and suppose that $G$ contains a cycle $C$ of length $\ell\geq 2t$. Set $I=I_{\NM_t(G)}$ and $n=|E(G)|$. Then 
	$$\reg I=2t-1,\quad
	\pd I\geq n-2t+1,\quad\text{and}\quad
	\depth I\leq 2t-1.$$
Moreover, for every $\ell-2t+1\leq i\leq n-2t+1$,	we have $$\beta_{i, i+2t-1}(I) \geq \binom{n-\ell}{i+2t-1-\ell}.$$ Consequently, $t_0(I)=t$ and $t_i(I)=i+2t-1$ for $\ell-2t+1\leq i\leq n-2t+1$.
\end{coro}

\begin{proof}
Since the cycle $C$ contains a $t$-matching, we have $I\ne0$. Hence \eqref{eq:regularity-leray} and Corollary~\ref{thm:cycle-persistence} yield $$\reg I=L_K(\NM_t(G))+1=2t-1.$$

To prove the asserted lower bound on the graded Betti numbers, fix an integer $i$ such that $\ell-2t+1\leq i\leq n-2t+1$. For every $W\subseteq E(G)$ satisfying $E(C)\subseteq W$ and $|W|=i+2t-1$, Corollary~\ref{thm:cycle-persistence} and Hochster's formula \eqref{eq:hochster} yield
$$\beta_{i, W}(I)=\dim_K\widetilde H_{2t-3}(\NM_t(G_W); K)\neq 0.$$
Such a set $W$ is obtained by choosing $i+2t-1-\ell$ edges from $E(G)\setminus E(C)$. Hence there are exactly $$\binom{n-\ell}{i+2t-1-\ell}$$ such sets. Since each of these sets $W$ has cardinality $i+2t-1$ and satisfies $\beta_{i,W}(I)\ne 0$, the definition of the graded Betti numbers gives
$$\beta_{i,i+2t-1}(I) \geq \binom{n-\ell}{i+2t-1-\ell}.$$
Since the minimal generators of $I=I_{\NM_t(G)}$ correspond to the $t$-matchings of $G$, they all have degree $t$. Hence $t_0(I)=t$. Moreover, the preceding Betti number bound gives $\beta_{i,i+2t-1}(I)>0$ for $\ell-2t+1\le i\le n-2t+1$. Consequently, $$i+2t-1\le t_i(I)\le i+\reg I=i+2t-1.$$ Thus $t_i(I)=i+2t-1$ throughout this range.

Finally, taking $i=n-2t+1$ in the Betti number bound yields $\beta_{n-2t+1,n}(I)\ne0$, hence $\pd I\ge n-2t+1$. The Auslander--Buchsbaum formula \cite{BH98}*{Theorem~1.3.3} therefore gives $\depth I=n-\pd I\le 2t-1$, as desired.
\end{proof}

In particular, if $G$ contains a $2t$-cycle, then $t_0(I)=t$ and $t_i(I)=i+2t-1$ for $1\leq i\leq n-2t+1$, so the maximal shifts are determined through homological degree $n-2t+1$.

%%%%%%%%%%%%%%%%%%%%%%%%%%%%%%%%%%%%%%%%%%%%%%%%%%%%%%%%

\section{Complete bipartite graphs and chessboard complexes}

Throughout this section, we assume that $2\leq t\leq r\leq s$ and use the notation $K_{r,s}$, $R_{r,s}$, $J_{t;r,s}$, and $\Delta_{r, s}$ introduced in the Introduction. Identify the edge set $E(K_{r, s})$ with $[r]\times[s]$, where $[m]=\{1,\ldots, m\}$. Under this identification, matchings correspond to placements of pairwise non-attacking rooks on an $r\times s$ chessboard. For further background on matching and chessboard complexes, see \cite{Wachs03}.

As noted in the Introduction, $J_{t;r,s}=F(\Delta_{r,s}^{[t-1]})$. In particular, $\Delta_{r,s}$ is pure of dimension $r-1$ and $J_{r;r,s}=F(\Delta_{r,s})$. Finally, for a nonzero proper homogeneous ideal $Q$ in a polynomial ring $S$, we use the standard relations $\reg(S/Q)=\reg Q-1$ and $\depth(S/Q)=\depth Q-1.$

Since $K_{r,s}$ contains a cycle of length $2t$, Corollary~\ref{cor:algebraic-consequences} and the preceding identities imply that
$$ \reg(R_{r,s}/J_{t;r,s})=2t-2
\quad\text{and}\quad
\depth(R_{r,s}/J_{t;r,s})\leq 2t-2.$$
For $t=r\geq 2$, combining these formulas with the depth lower bound from \cite{JZWZ23}*{Corollary~5.8} establishes the equalities conjectured in \cite{JZWZ23}*{Conjecture~5.15}.

More generally, we will prove that $$\depth(R_{r,s}/J_{t;r,s})=2t-2$$ for every $2\leq t\leq r$. Equivalently, $\depth J_{t;r,s}=2t-1$. Our proof extends the induction used in \cite{JZWZ23}*{Theorem~5.7} to include the case $t<r$. The induction involves sums of rook ideals with varying numbers of rooks and specified columns excluded. We therefore introduce the following notation.

For positive integers $a, b$, set $R_{a, b}=K[x_{ij}:1\leq i\leq a,\ 1\leq j\leq b]$. For $C\subseteq[b]$ and $p\geq 0$, let $J_p^C(a, b)\subseteq R_{a, b}$ be the ideal generated by the monomials corresponding to placements of $p$ non-attacking rooks avoiding the columns in $C$:
$$J_p^C(a, b)=\left(x_{i_1j_1}\cdots x_{i_p j_p}: \substack{
	i_1,\ldots, i_p\in[a]\text{ are pairwise distinct},\\
	j_1,\ldots, j_p\in[b]\setminus C\text{ are pairwise distinct}
}\right).$$
We use the conventions $J_0^C(a, b)=R_{a, b}$ and $J_p^C(a, b)=0$ if $p>a$ or $p>b-|C|$. With the notation above, we have $J_{t; r, s}=J_t^\varnothing(r, s)$.

\begin{lem}\label{lem:zhu-general}
Let $I=\sum_{\lambda=1}^q J_{p_\lambda}^{C_\lambda}(a, b)\subseteq R_{a, b},$ where $q\geq 1$, $C_\lambda\subseteq[b]$, $p_\lambda\geq 1$, and every summand is nonzero. Put $p_0=\min_{1\leq \lambda \leq q}p_\lambda$. Then
$$\depth(R_{a, b}/I)\geq 2p_0-2.$$
\end{lem}

\begin{proof}
We prove the assertion by induction on $a$, simultaneously for all $b$, all $q$, and all admissible collections $\{(p_\lambda,C_\lambda)\}_{\lambda=1}^q$. If $p_0=1$, the required lower bound is $0$, and the assertion follows because $I$ is a proper ideal. We may therefore assume that $p_0\geq 2$. Since every summand is nonzero, necessarily $a\geq p_0\geq 2$. For $U\subseteq[b]$, set
$$x_{a, U}=\prod_{u\in U}x_{au}, \quad P_{U^c}=(x_{av}: v\notin U).$$
By \cite{JZWZ23}*{Corollary~2.12}, applied with $V=\{x_{a1},\ldots, x_{ab}\}$ as in the proof of \cite{JZWZ23}*{Theorem~5.7}, it is enough to prove
$$\depth\frac{R_{a,b}}{(I:x_{a,U})+P_{U^c}} \geq 2p_0-2 \quad \text{for every }U\subseteq[b].$$
Set $R'=K[x_{ij}:1\leq i<a,\ 1\leq j\leq b]$, and regard ideals of $R'$ as their extensions to $R_{a, b}$. Fix $p\geq 2$ and $C\subseteq [b]$. Separating the monomial generators of $J_p^C(a, b)$ according to whether they avoid row $a$ or use $x_{av}$ for some $v\in[b]\setminus C$ gives
$$J_p^C(a,b) =J_p^C(a-1, b)+\sum_{v\in[b]\setminus C} x_{av}J_{p-1}^{C\cup\{v\}}(a-1,b).$$ 
Taking the colon by $x_{a,U}$ and then adding $P_{U^c}$ yields
$$(J_p^C(a, b): x_{a,U})+P_{U^c}=P_{U^c}+J_p^C(a-1,b)+\sum_{u\in U\setminus C}J_{p-1}^{C\cup\{u\}}(a-1,b).$$
For every $v\in[b]\setminus C$, we have $J_p^C(a-1,b)\subseteq J_{p-1}^{C\cup\{v\}}(a-1,b)$. Indeed, for each generator of $J_p^C(a-1,b)$, delete the rook in column $v$ if it occurs, and otherwise delete any one of the rooks. The resulting monomial is a generator of $J_{p-1}^{C\cup\{v\}}(a-1,b)$. Consequently, if $U\not\subseteq C$, then $U\setminus C\neq\varnothing$, so the term $J_p^C(a-1,b)$ in the preceding equality is redundant. Therefore
\begin{equation}\tag{*}\label{eq:single-rook-colon}
	(J_p^C(a,b): x_{a,U})+P_{U^c}
	=P_{U^c}+\varepsilon_U^C J_p^C(a-1,b) +\sum_{u\in U\setminus C} J_{p-1}^{C\cup\{u\}}(a-1,b),
\end{equation}
where $\varepsilon_U^C=1$ if $U\subseteq C$, and $\varepsilon_U^C=0$ otherwise.

Since colon by a monomial distributes over sums of monomial ideals, \eqref{eq:single-rook-colon} yields
$$(I:x_{a,U})+P_{U^c}=P_{U^c}+H_U,$$ where $H_U\subseteq R'$ is a sum of ideals of the form $J_p^C(a-1,b)$; the sum may be zero. After zero summands are omitted, every parameter $p$ occurring in $H_U$ satisfies $p\geq p_0-1$. If $H_U=0$, then
$$R_{a, b}/P_{U^c}\cong R'[x_{au}:u\in U]$$
has depth $(a-1)b+|U|$. Since every summand of $I$ is nonzero, $a, b\geq p_0$, and hence
$$(a-1)b+|U|\geq (p_0-1)p_0\geq 2p_0-2.$$
We may therefore assume that $H_U\neq 0$.

If $U=\varnothing$, then the quotient is isomorphic to $R'/H_U$. In this case, \eqref{eq:single-rook-colon} gives $H_U=\sum_{\lambda=1}^q J_{p_\lambda}^{C_\lambda}(a-1,b)$. Since $H_U\neq 0$, we necessarily have $p_0<a$. 
Thus, a summand with parameter $p_0$
remains nonzero after deleting the last row. Consequently, $p_0$ is the least parameter appearing in $H_U$, and the induction hypothesis  yields $$\depth(R'/H_U)\geq 2p_0-2.$$

Now assume that $U\neq \varnothing$, and let $p_U$ be the least parameter occurring in $H_U$. Since the variables $x_{au}$ with $u\in U$ are free, it follows from the induction hypothesis that
$$\depth\frac{R_{a,b}}{P_{U^c}+H_U} = |U|+\depth(R'/H_U) \geq |U|+2p_U-2.$$
Since $p_U\geq p_0-1$, the required bound follows unless $U=\{u\}$ and $p_U=p_0 -1$.

It remains to consider this case. By \eqref{eq:single-rook-colon}, we see that every nonzero summand of $H_U$ is either of the form
$J_{p_\lambda-1}^{C_\lambda\cup\{u\}}(a-1,b)$ or of the form $J_{p_\lambda}^{C_\lambda}(a-1,b)$ with $u\in C_\lambda$. Hence every summand of $H_U$ avoids column $u$. Since $P_{U^c}$ contains all variables in row $a$ except $x_{au}$, we obtain that none of the variables $x_{1u}, \ldots, x_{au}$ occurs in a minimal monomial generator of $P_{U^c}+H_U$. Set $$R''=K[x_{ij}:1\leq i<a,\ j\in[b]\setminus\{u\}].$$ 
After deleting column $u$ and relabelling the remaining columns, $H_U$ is the extension of a nonzero sum $H'_U\subseteq R''$ of ideals of the form $J_p^C(a-1,b-1)$, whose least parameter is $p_U=p_0-1$. Therefore
$$\frac{R_{a, b}}{P_{U^c}+H_U} \cong (R''/H'_U)[x_{1u}, \ldots, x_{au}].$$
The induction hypothesis yields
$$\depth\frac{R_{a, b}}{P_{U^c}+H_U} \geq a+2(p_0-1)-2 \geq 2p_0- 2,$$ since $a \geq 2$. This completes the induction.
\end{proof}

The depth estimate required above is now an immediate consequence.

\begin{coro}\label{cor:zhu-bound}
	For $2\leq t\leq r\leq s$, $\depth J_{t; r, s}\geq 2t-1$.
\end{coro}

\begin{proof}
Applying Lemma~\ref{lem:zhu-general} with $q=1$, $a=r$, $b=s$, $p_1=t$, and $C_1=\varnothing$ gives $\depth(R_{r, s}/J_{t; r, s})\geq 2t-2$. By the standard depth relation recalled above,
$$\depth J_{t; r, s}=\depth(R_{r, s}/J_{t; r, s})+1 \geq 2t-1.$$
\end{proof}

\begin{rem}
	Lemma~\ref{lem:zhu-general} also recovers \cite{JZWZ23}*{Theorem~5.7}. Indeed, for $A\in\binom{[b]}{a}$, let $D_A$ be the chessboard complex on the $a\times a$ subboard determined by the columns in $A$. Viewing $F(D_A)$ as an ideal of $R_{a, b}$ by extension, we have $F(D_A)=J_a^{[b]\setminus A}(a, b)$. Therefore, applying Lemma~\ref{lem:zhu-general} with $p_\lambda=a$ and $C_\lambda=[b]\setminus A_\lambda$ recovers \cite{JZWZ23}*{Theorem~5.7}.
\end{rem}

Combining Corollaries~\ref{cor:algebraic-consequences} and~\ref{cor:zhu-bound}, we determine the depth, projective dimension, and all maximal shifts of $J_{t;r,s}$. We also compute its unique extremal Betti number, which occurs at the maximal shift in the last homological degree.

Recall that a nonzero graded Betti number $\beta_{i,i+j}(J_{t;r,s})$ is called \emph{extremal} if $\beta_{p,p+q}(J_{t;r,s})=0$ for all $(p,q)\ne(i,j)$ with $p\geq i$ and $q\geq j$; see \cite{BCP99}. The complete statement is as follows.

\begin{thm}\label{thm:complete-bipartite}
Let $2 \leq t \leq r \leq s$. Then $\reg J_{t; r, s}=\depth J_{t; r, s}=2t-1$ and $\pd J_{t; r, s}=rs-2t+1.$ Moreover, $t_0(J_{t; r, s})=t$ and $t_i(J_{t; r, s})=i+2t-1$ for $1\leq i \leq rs-2t+1$. Finally,
$$\beta_{rs-2t+1, rs}(J_{t;r,s})=\binom{r-1}{t-1}\binom{s-1}{t-1},$$ and this is the unique extremal Betti number of $J_{t; r, s}$.
\end{thm}

\begin{proof}
Let $n=rs$. Since $K_{r,s}$ contains a cycle of length $2t$ and $J_{t;r,s}=I_{\NM_t(K_{r,s})}$, Corollary~\ref{cor:algebraic-consequences} yields $$\reg J_{t;r,s}=2t-1,\quad t_0(J_{t;r,s})=t,$$ and $$t_i(J_{t;r,s})=i+2t-1 \quad (1\leq i\leq n-2t+1).$$ In particular, $\beta_{n-2t+1,n}(J_{t;r,s})\ne0$, so $\pd J_{t;r,s}\geq n-2t+1$.
	
On the other hand, Corollary~\ref{cor:zhu-bound} and the Auslander--Buchsbaum formula yield $$n-2t+1\leq\pd J_{t;r,s}=n-\depth J_{t;r,s}\leq n-2t+1.$$
Thus $\pd J_{t;r,s}=n-2t+1$ and $\depth J_{t;r,s}=2t-1$.
	
Finally, since $J_{t;r,s}$ is squarefree, its multigraded Betti numbers are supported in squarefree multidegrees. In total degree $n$, the only such multidegree is $E(K_{r,s})$. Hence the definition of the graded Betti numbers and Hochster's formula~\eqref{eq:hochster} give
$$\beta_{n-2t+1,n}(J_{t;r,s})
	=\beta_{n-2t+1,E(K_{r,s})}(J_{t;r,s})
	=\dim_K\widetilde H_{2t-3} (\NM_t(K_{r,s});K).$$
By \cite{LSW08}*{Theorem~1.4}, we have $$\NM_t(K_{r,s})\simeq \bigvee_{\binom{r-1}{t-1}\binom{s-1}{t-1}} S^{2t-3}.$$ Therefore,
$$\beta_{n-2t+1,n}(J_{t;r,s})=\binom{r-1}{t-1}\binom{s-1}{t-1}.$$
Since every nonzero graded Betti number $\beta_{p,p+q}(J_{t;r,s})$ satisfies $p\leq n-2t+1$ and $q\leq 2t-1$, this is the unique extremal Betti number.
\end{proof}

%%%%%%%%%%%%%%%%%%%%%%%%%% Funding and Declarations %%%%%%%%%%%%%%%%%%%%%%%%%%

\medskip
\noindent\textbf{Funding}\quad This research was partially supported by a research project from the Liaoning Provincial Department of Education (Grant No. LJ232510166002) for the first author and the National Natural Science Foundation of China (Grant No. 11971338) for the second author.

\medskip
\noindent\textbf{Data availability}\quad No datasets were generated or analyzed during this study.

\section*{Declarations}

\noindent\textbf{Conflict of interest}\quad On behalf of all authors, the corresponding author states that there is no conflict of interest.

%%%%%%%%%%%%%%%%%%%%%%%%%%References%%%%%%%%%%%%%%%%%%%%%%%%

\end{document}